\documentclass[12pt]{article}
\usepackage{amssymb, amsmath, amsthm, graphicx, hyperref,xcolor,verbatim}
\newcommand{\SL}{{\text {\rm SL}}}
\newcommand{\R}{\mathbb{R}}
\newcommand{\Z}{\mathbb{Z}}

\newcommand{\C}{\mathbb{C}}

\newtheorem{thm}{Theorem}

\newtheorem{prop}[thm]{Proposition}
\newtheorem{lemma}[thm]{Lemma}
\newtheorem*{rem}{Remark}
\definecolor{darkgreen}{RGB}{0,96,0}

\begin{document}

\title{Modular Forms with Only Nonnegative Coefficients}
\author{Paul Jenkins and Jeremy Rouse}
\date{\today}

\maketitle

\begin{abstract}We study modular forms for $\SL_2(\Z)$ with no negative Fourier coefficients.  Let $A(k)$ be the positive integer where if the first $A(k)$ Fourier coefficients of a modular form of weight $k$ for $\SL_2(\Z)$ are nonnegative, then all of its Fourier coefficients are nonnegative, so that $A(k)$ can be interpreted as a ``nonnegativity Sturm bound''.  We give upper and lower bounds for $A(k)$, as well as an upper bound on the $n$th Fourier coefficient of any form with no negative Fourier coefficients.\end{abstract}

\section{Introduction}

Modular forms naturally arise as generating functions which count partitions, representations of an integer by a quadratic form, elements in class groups, or other arithmetically interesting objects.  If the coefficients of a modular form
\[
  f(z) = \sum_{n=0}^{\infty} a(n) q^{n}, \quad q = e^{2 \pi i  z}
\]
count something, then necessarily $a(n) \in \Z$ and $a(n) \geq 0$ for all $n$. The first condition may be checked using Sturm's theorem \cite{Sturm}, while the second condition is less straightforward to check. We study the problem
of determining when holomorphic modular forms for $\SL_{2}(\Z)$ have nonnegative Fourier coefficients.

Nonzero modular forms for $\SL_2(\Z)$ with nonnegative Fourier coefficients must have weight $k \equiv 0 \pmod{4}$.
Every modular form for $\SL_2(\Z)$ is a linear combination of an Eisenstein series and a cusp form, and every cusp form has both positive and negative Fourier coefficients.  Thus, any nonzero form $f(z)$ with no negative Fourier coefficients must have $a(0) > 0$, and the Fourier coefficients of $f(z)$ will be affected by those of the Eisenstein series \[E_{k}(z) = 1 - \frac{2k}{B_{k}} \sum_{n=1}^{\infty} \sigma_{k-1}(n) q^{n}.\]
For weights $k \equiv 2 \pmod{4}$, the Eisenstein coefficients are negative for $n > 0$ and asymptotically greater than the cusp form coefficients, so if $n$ is sufficiently large then $a(n) < 0$. Thus, any modular form of weight $k \equiv 2 \pmod{4}$ will have negative Fourier coefficients.

If $f(z) = \sum a(n) q^{n}$ is a modular form for $\SL_2(\Z)$, define $\mathcal{N}(f)$ to be the smallest integer $r$ for which $a(r) < 0$, if such an $r$ exists. For forms $f$ with nonnegative coefficients, define $\mathcal{N}(f) = \infty$. For nonnegative weights $k \equiv 0 \pmod{4}$, define $A(k) = \max \{ \mathcal{N}(f) : f \in M_{k} \text{ and } \mathcal{N}(f) < \infty \}$. In other words, 
for a modular form $f(z)$ of weight $k$, if each of the leading Fourier coefficients of $f(z)$ up through the $q^{A(k)}$ term is nonnegative,
then all of the Fourier coefficients of $f(z)$ are nonnegative, and $A(k)$ can be interpreted as a ``nonnegativity Sturm bound'' for modular forms of weight $k$ for $\SL_2(\Z)$.  We will show in Section~\ref{WellDefined} that such a maximum always exists, so $A(k)$ is well-defined.

The quantity $A(k)$ is connected with sign changes in Fourier expansions of cusp forms. In particular, suppose that $g(z) = \sum_{n=1}^{\infty} b(n) q^{n}$ is a weight $k$ cusp form normalized so that the leading coefficient is positive. Assume that for some positive integer $C_1$, we have $b(n) > 0$
for all $1 \leq n \leq C_1-1$ and $b(C_1) < 0$. Then for sufficiently small $\epsilon > 0$, $g(z) + \epsilon E_{k}(z)$ will be a weight $k$ modular form whose first $C_1-1$ Fourier coefficients are positive and for which
the coefficient of $q^{C_1}$ is negative. As a consequence, $A(k) \geq C_1$. In 2018, He and Zhao showed \cite{HeZhao} that if $g$ is a cusp form of weight $k$ and squarefree level $N$, there are positive integers $n_{1}, n_{2} \ll (k N)^{2 + \epsilon}$ with $b(n_{1}) b(n_{2}) < 0$. In 2023, Cho, Jin and Lim \cite{ChoJinLim} proved the bound $k^{6+\epsilon} N^{9+\epsilon}$ for general level.

For the rest of this paper, we assume $k \equiv 0 \pmod{4}$ and normalize a form $f(z)$ with only nonnegative coefficients to have the Fourier expansion
\[f(z) = 1 + \sum_{n=1}^\infty a(n) q^n. \]
The main result of this paper provides
lower and upper bounds on the quantity $A(k)$.
\begin{thm}
\label{Akbounds}
For any positive integer $k \equiv 0 \pmod{4}$ with $k \geq 12$, we have
\[
  \frac{(k-1)^{2}}{16 \pi^{2}} < A(k) \leq \frac{1}{7316} k^{4} (\log k + \log \log k)^{2}.
\]
\end{thm}

To prove the lower bound in Theorem~\ref{Akbounds} we show that there is
a nonzero cusp form, namely the Poincar\'e series $P_{1}(z)$, whose
first sign change occurs after $\lfloor \frac{(k-1)^{2}}{16 \pi^{2}}\rfloor$ terms in its Fourier expansion. This shows
that the result of He and Zhao mentioned above is optimal in the weight aspect.

The table below shows the lower bound $L(k) := \lceil (k-1)^{2}/(16 \pi^{2}) \rceil$,
the true value of $A(k)$, and the upper bound $U(k) :=
\lfloor (k^{4} (\log k + \log \log k)^{2})/7316 \rfloor$ for the twenty smallest nontrivial weights $12 \leq k \leq 88$.  For $k$ in this range, $A(k)$ appears to be very close to the lower bound, and we conjecture that $A(k) = O(k^2)$.

\begin{table}
\begin{center}
\caption{Bounds from Theorem~\ref{Akbounds} and the true size of $A(k)$ for $12 \leq k \leq 88$}
\label{table1}
\begin{tabular}{cccc|cccc}
$k$ & $L(k)$ & $A(k)$ & $U(k)$ & $k$ & $L(k)$ & $A(k)$ & $U(k)$\\
\hline
$12$ & $1$ & $2$ & $32$ & $52$ & $17$ & $22$ & $28341$\\
$16$ & $2$ & $3$ & $128$ & $56$ & $20$ & $26$ & $39459$\\
$20$ & $3$ & $4$ & $366$ & $60$ & $23$ & $29$ & $53663$\\
$24$ & $4$ & $6$ & $851$ & $64$ & $26$ & $33$ & $71508$\\
$28$ & $5$ & $8$ & $1728$ & $68$ & $29$ & $36$ & $93600$\\
$32$ & $7$ & $10$ & $3177$ & $72$ & $32$ & $41$ & $120598$\\
$36$ & $8$ & $12$ & $5422$ & $76$ & $36$ & $45$ & $153217$\\
$40$ & $10$ & $14$ & $8727$ & $80$ & $40$ & $49$ & $192226$\\
$44$ & $12$ & $16$ & $13403$ & $84$ & $44$ & $54$ & $238450$\\
$48$ & $14$ & $19$ & $19807$ & $88$ & $48$ & $59$ & $292773$
\end{tabular}
\end{center}
\end{table}

To prove the upper bound, we use the modularity relation $f\left(-\frac{1}{z}\right) = z^{k} f(z)$, which implies that if $y > 0$ is a real number, then
\[
  1 + \sum_{n=1}^{\infty} a(n) \left(e^{-2 \pi n / y} - y^{k} e^{-2\pi n y}\right)
  = y^{k}.
\]
For $y \approx \frac{k}{2 \pi} \log k$, the expression $e^{-2 \pi n / y} - y^{k} e^{-2 \pi n y}$ is positive for all positive integers $n$. Using this expression, it is possible to bound $a(n)$ for relatively small values of $n$. Next, we rely on the main result of \cite{JenkinsRouse}, which gives an explicit bound on all the Fourier coefficients of a weight $k$ cusp form in terms of the coefficients of $q^{1}$, $\ldots$, $q^{\ell}$, where $\ell$ is the dimension of the space of cusp forms of weight $k$.
Combining these results yields the upper bound on $A(k)$.
This argument also produces explicit bounds on $a(n)$ for $n \leq \ell$
under the assumption that all of the Fourier coefficients are nonnegative.
The main result of \cite{JenkinsRouse} then allows us to give a nontrivial \emph{upper} bound on the Fourier coefficients of a modular form under the assumption that its coefficients are nonnegative.
\begin{thm}
\label{upperbound}
Suppose that $k \geq 16$ and $f(z) = 1 + \sum_{n=1}^{\infty} a(n) q^{n} \in M_{k}$ and $a(n) \geq 0$ for all positive integers $n$. Then
\[
  a(n) \leq -\frac{2k}{B_{k}} \sigma_{k-1}(n) + 1.07 \cdot 10^{-10}
  \left(\frac{k}{2 \pi}\right)^{k} (\log k)^{1.5k} d(n) n^{\frac{k-1}{2}}.
\]
\end{thm}
A similar result with a larger constant is true in the case that $k = 12$. In particular, we show in Section~\ref{smallk} that if $f(z) = 1 + \sum_{n=1}^{\infty} a(n) q^{n} \in M_{12}$ has nonnegative coefficients, then
\[
  a(n) \leq -\frac{24}{B_{12}} \sigma_{11}(n) + 8096 d(n) n^{\frac{11}{2}}.
\]

For a form $f(z) = 1 + \sum_{n=1}^{\infty} a(n) q^{n}$, the Eisenstein contribution to $a(n)$ will dominate when $n$ is sufficiently large. Our next result quantifies when this occurs, under the assumption that the first few Fourier coefficients are nonnegative.
\begin{thm}
\label{positivity}
Assume that $f(z) = 1 + \sum_{n=1}^{\infty} a(n) q^{n} \in M_{k}$ satisfies
\[
  a(n) \geq 0 \text{ for all } 1 \leq n \leq k^{2} (\log k + \log \log k)^{2}.
\]
Then
\[
a(n) > 0 \text{ for all } n \geq \frac{1}{7316} k^{4} (\log k + \log \log k)^{2}.
\]
\end{thm}

If the modular form $f(z) = 1 + \sum_{n=1}^{\infty} a(n) q^{n}$ is a generating function, the result above addresses the question of how large $n$ must be in order to guarantee that $a(n) > 0$. In particular, if $Q(\vec{x}) = \frac{1}{2} \vec{x}^{T} A \vec{x}$ for a positive-definite matrix $A \in {\rm GL}_{2k}(\Z)$ for which $A$ and $A^{-1}$ both have even diagonal entries, then
\[ \theta_{Q}(z) = \sum_{\vec{x} \in \Z^{2k}} q^{Q(\vec{x})}
\]
is a weight $k$ modular form for $\SL_{2}(\Z)$ with nonnegative Fourier coefficients and Theorem~\ref{positivity} shows that $Q$ represents
every positive integer $n \geq \frac{1}{7316} k^{4} (\log k + \log \log k)^{2}$.

\begin{rem}
Theorem~\ref{positivity} applies to any modular form of weight $k$. There are more specialized arguments that apply to a theta series that should yield much better bounds.
\end{rem}

It is natural to study analogous questions for modular forms of level $N > 1$. For example, the space $M_{2}(\Gamma_{0}(4))$ is spanned by
\begin{align*}
  F_{1}(z) = 2 E_{2}(2z) - E_{2}(z) &= 1 + 24 \sum_{n=1}^{\infty} \left(\sum_{\substack{d \mid n \\ d~\text{odd}}} d\right) q^{n}, \text{ and } \\\textbf{}
  F_{2}(z) = \frac{\eta^{8}(4z)}{\eta^{4}(2z)} &= \sum_{n~\text{odd}} \sigma(n) q^{n}.
\end{align*}
A modular form $f(z) = c_{1} F_{1}(z) + c_{2} F_{2}(z)$ in $M_{2}(\Gamma_{0}(4))$ has nonnegative coefficients and leading coefficient $1$ if and only if
$c_{1} = 1$ and $c_{2} \geq 0$. For this reason, no analogue of Theorem~\ref{upperbound} is true in this instance.

A greater challenge occurs in the space $M_{2}(\Gamma_{0}(22))$. We have
$F_{1}(z) = 2E_{2}(2z) - E_{2}(z) \in M_{2}(\Gamma_{0}(22))$ and
$G(z) = \eta^{2}(z) \eta^{2}(11z) \in M_{2}(\Gamma_{0}(22)) = \sum_{n=1}^{\infty} c(n) q^{n}$. This latter form is a Hecke eigenform
and so $|c(n)| \leq d(n) \sqrt{n} \leq 2n$.
For any $\epsilon > 0$, the form $F_{1}(z) + \epsilon G(z)$ has leading coefficient $1$, and the bound above implies that if the $n$th coefficient is negative, then $n > 12/\epsilon$. However, knowing that $c(1) = 1$, $c(2) = -2$
and $c(2^{n}) = -2c(2^{n-1}) - 2c(2^{n-2})$ implies that
\[
  c(2^{n}) = \begin{cases}
    2^{n/2} & \text{ if } n \equiv 0, 2 \pmod{8}\\
    -2^{(n+1)/2} & \text{ if } n \equiv 1 \pmod{8}\\
    0 & \text{ if } n \equiv 3,7 \pmod{8}\\
    -2^{n/2} & \text{ if } n \equiv 4, 6 \pmod{8}\\
    2^{(n+1)/2} & \text{ if } n \equiv 5 \pmod{8}.\\
    \end{cases}
\]
Since the coefficient of $q^{2^{n}}$ in $F_{1}(z)$ is $24$ for all $n$,
this implies that for any $\epsilon > 0$ there is a positive integer $n$
so that the coefficient of $q^{2^{n}}$ in $F_{1}(z) + \epsilon G(z)$ is negative. In particular, for any $\epsilon > 0$, the quantity $\mathcal{N}(F_{1}(z) + \epsilon G(z))$ is finite, but larger than $12/\epsilon$. Hence there is no analogue of Theorem~\ref{Akbounds}, and no ``nonnegativity Sturm bound'' for $M_{2}(\Gamma_{0}(22))$.

While there are challenges generalizing the results above to higher
level, there are potential applications to theta series of more general quadratic forms, generating functions whose coefficients are sums of class numbers, and various partition generating functions.

The remainder of the paper proceeds as follows.  In Section~\ref{back}, we review relevant background. In Section~\ref{smallk}, we explicitly compute $A(k)$ for $12 \leq k \leq 88$, and in Section~\ref{WellDefined} we show that $A(k)$ is well-defined. In Section~\ref{lowerbound}, we prove the lower bound in Theorem~\ref{Akbounds}, and in Section~\ref{finalsec}, we prove the upper bound in Theorem~\ref{Akbounds} and prove Theorems~\ref{positivity} and~\ref{upperbound}.

\vspace{.15in}
\noindent\textbf{Acknowledgements.} The authors thank Nick Andersen, Pace Nielsen, Martin Raum, and the anonymous referees for helpful conversations and feedback.

\vspace{.15in}
\noindent\textbf{Data Availability Statement.} The authors declare that the data relevant to this work is contained in this paper, namely in Table~\ref{table1} and Table~\ref{table2}.

\section{Background}
\label{back}

In this section, we recall definitions and previously known theorems that will be used later in the paper.

We let $M_{k}$ denote, as usual, the $\C$-vector space of modular forms of weight $k$ for $\SL_{2}(\Z)$, and let $S_{k}$ denote the subspace of cusp forms. The usual weight $k$ Eisenstein series is
\[
  E_{k}(z) = 1 - \frac{2k}{B_{k}} \sum_{n=1}^{\infty} \sigma_{k-1}(n) q^{n},
\]
where $B_{k}$ is the $k$th Bernoulli number and $\sigma_{k-1}(n)$ is the sum of the $(k-1)$st powers of the divisors of $n$. If $f(z) = \sum_{n=0}^{\infty} a(n) q^{n}$, then $f(z) - a(0) E_{k}(z) \in S_{k}$.

We note that when $k \equiv 0 \pmod{4}$,
\[
  -\frac{2k}{B_{k}} = \frac{(2 \pi)^{k}}{(k-1)! \zeta(k)}.
\]
Because of the presence of the factorial in the denominator, we will require lower bounds on $n!$. In \cite{Robbins}, Robbins shows that for all $n \geq 1$, we have $\sqrt{2 \pi n} \left(\frac{n}{e}\right)^{n} e^{\frac{1}{12n+1}} < n! < \sqrt{2 \pi n} \left(\frac{n}{e}\right)^{n} e^{\frac{1}{12n}}$, which we will use in the simplified form
\begin{equation}
\label{facbound}
  \sqrt{2 \pi n} \left(\frac{n}{e}\right)^{n} < n! < \sqrt{2 \pi n} \left(\frac{n}{e}\right)^{n} e^{\frac{1}{12n}}.
\end{equation}

We will make use of the standard modular forms
\[
  \Delta(z) = q \prod_{n=1}^{\infty} (1-q^{n})^{24} = \sum_{n=1}^{\infty} \tau(n) q^{n} \in S_{12}
\]
and
\[
  j(z) = \frac{E_{4}(z)^{3}}{\Delta(z)} = q^{-1} + 744 + 196884q + \cdots.
\]

If $k$ is an even positive integer, the Poincar\'e series
of index $m \geq 1$ (see \cite[p. 51]{Iwaniec}) has Fourier expansion
\begin{equation}
\label{poincare}
  P_{m}(z)
  = \sum_{n=1}^{\infty}
  \left(\frac{n}{m}\right)^{\frac{k-1}{2}}
  \left(\delta_{mn} + 2 \pi i^{-k} \sum_{c=1}^{\infty}
  \frac{K(m,n;c)}{c} J_{k-1}\left(\frac{4 \pi \sqrt{mn}}{c}\right)\right) q^{n}.
\end{equation}
Here $J_{k-1}$ is a Bessel function of the first kind, $K(m,n;c)$ is the usual Kloosterman sum
\[
  K(m,n;c) = \sum_{x \in (\Z/c\Z)^{\times}}
  e^{2 \pi i \cdot \frac{mx+nx^{-1}}{c}}, \quad \text{ and } \quad \delta_{mn} = \begin{cases} 1 & \text{ if } m = n,\\ 0 & \text{ otherwise. } \end{cases}
\]
For all $m \geq 1$, $P_{m}(z) \in S_{k}$.

Next, we state the main result of \cite{JenkinsRouse} which gives a bound on all of the Fourier coefficients of a cusp form of weight $k$ in terms of the first $\ell$ coefficients, where $\ell = \lfloor \frac{k}{12}\rfloor$. We use $d(n)$ to denote the number of divisors of the positive integer $n$.

\begin{thm}
\label{cuspformbound}
Suppose that $\sum_{n=1}^{\infty} c(n) q^{n} \in S_{k}$. Then
\[
  |c(n)| \leq \sqrt{\log k} \left(11 \cdot \sqrt{\sum_{m=1}^{\ell} \frac{|c(m)|^{2}}{m^{k-1}}} + \frac{e^{18.72} (41.41)^{k/2}}{k^{(k-1)/2}} \cdot
  \left| \sum_{m=1}^{\ell} c(m) e^{-7.288m} \right|\right) \cdot d(n) n^{\frac{k-1}{2}}.
\]
\end{thm}

\section{Computing $A(k)$ for small values of $k$}
\label{smallk}

As above, for any positive weight $k \equiv 0 \pmod{4}$ let $\ell = \lfloor \frac{k}{12}\rfloor$, so that the space $M_k$ has dimension $\ell+1$.  There is a basis for $M_k$ consisting of modular forms $F_{k, m}$ for $0 \leq m \leq \ell$ where each $F_{k, m}(z)$ has a Fourier expansion of the form \[F_{k, m}(z) = q^{m} + O(q^{\ell+1}) = \sum_{n=0}^{\infty} c_{m}(n) q^{n}.\]  We use these basis elements, whose Fourier coefficients $c_m(n)$ are straightforward to compute, to find $A(k)$ for small values of $k$.

For $k=12$, a basis for $M_{12}$ is
\begin{align*}
  F_{12, 0} &= 1 + O(q^2),\\
  F_{12, 1} = \Delta &= q + O(q^2),
\end{align*}
and any normalized form $f = 1 + \sum_{n\geq 1} a(n) q^n \in M_{12}$ may be written \[f = F_{12, 0} + A F_{12, 1} = 1 + A q + (196560-24A)q^2 + \cdots.\]  If $f$ has only nonnegative Fourier coefficients, then the coefficients of $q$ and $q^2$ give the inequalities $0 \leq A$ and $A \leq 8190$ respectively.  Writing $f$ as $f = E_{12} + d_1 F_{12, 1}$, we find that $d_1 \leq 8096$.  Using the well-known bound $|\tau(n)| \leq d(n) n^{11/2}$ on the coefficients $\tau(n)$ of $\Delta = F_{12, 1}$, we find that the coefficient of the Eisenstein part is larger than the coefficient of the cusp form part for all $n \geq 3$, so $a(n) > 0$ for $n \geq 3$.  Thus, if $a(1)$ and $a(2)$ are nonnegative, then all $a(n)$ are nonnegative.  The form \[F_{12, 0} + 8191 F_{12, 1} = 1 + 8191q - 24q^2 + \cdots\] has all Fourier coefficients nonnegative except for the coefficient of $q^2$, so $A(12) = 2$.

The spaces $M_{16}$ and $M_{20}$ are likewise two-dimensional, and similar computations may be carried out. In weight $16$, the $q^3$ term of a form $f = F_{16, 0} + A F_{16, 1}$ with nonnegative Fourier coefficients gives $A \leq \frac{174080}{9}$, and writing $f = E_{16} + d_1 F_{16, 1}$ gives $d_1 \leq 19338$.  The $n$th Fourier coefficient of $F_{16, 1}$ is bounded by $d(n) n^{15/2}$, so the coefficient of the Eisenstein part dominates for $n \geq 4$.  The form \[F_{16, 0} + 19343 F_{16, 1} = 1+19343 q + 4324968 q^2 - 2604 q^3 + O(q^4)\] has every Fourier coefficient nonnegative except for the coefficient of $q^3$, so $A(16) = 3$.  We note that if $A \geq 0$, then the coefficient of $q^2$ in the Fourier expansion of $f$ is necessarily nonnegative.  Thus, if the coefficients of $q$ and $q^3$ are nonnegative in a form $f = 1 + \sum_{n \geq 1} a(n) q^n \in M_{16}$, then all other Fourier coefficients are nonnegative.  In weight $20$, similar computations give $0 \leq A \leq 65686.0145$ and $d_1 \leq 65686$, with the Eisenstein part dominating for $n \geq 5$.  The form \[F_{20, 0} + 65687 F_{20, 1} = 1+65687 q+29992872 q^2+3415037124 q^3-311824 q^4 + O(q^5)\] has every Fourier coefficient nonnegative except for the coefficient of $q^4$, so $A(20) = 4$.  The coefficients of $q^2$ and $q^3$ in $f$ will be nonnegative whenever $A \geq 0$, so checking the coefficients of $q$ and $q^4$ suffices to show nonnegativity of all coefficients for forms in $M_{20}$.

The space $M_{24}$ has dimension three, and is spanned by the canonical basis elements
\begin{align*}
    F_{24, 2} = \Delta^2 &= q^2 -48q^3 + 1080 q^4 - 15040 q^5 + \cdots,\\
    F_{24, 1} = \Delta^2(j - 696) &= q + 195660 q^3 + 12080128 q^4 + \cdots,\\
    F_{24, 0} = \Delta^2(j^2 -1440 j + 125280) &= 1 + 52416000 q^3 + 39007332000 q^4 + \cdots.
\end{align*}
If the form
\begin{align*}f(z) = 1 + \sum_{n=1}^{\infty} a(n) q^n &=  F_{24, 0} + A F_{24, 1} + B F_{24, 2} \\ &= 1 + A q + B q^2 + \cdots \in M_{24}\end{align*} has only nonnegative Fourier coefficients, then using the Fourier expansions of the basis elements to compute each Fourier coefficient gives a sequence of inequalities on $A$ and $B$.  The first few are
\[a(1) = A \geq 0,\]
\[a(2) = B \geq 0,\]
\[a(3) = 52416000+195660 A - 48 B \geq 0,\]
\[a(4) = 39007332000 + 12080128 A + 1080 B \geq 0,\]
\[a(5) = 6609020221440 + 44656110 A - 15040 B \geq 0,\]
\[a(6) = 437824977408000 - 982499328 A + 143820 B \geq 0.\]
The region of the $A$-$B$ plane satisfying all of these inequalities is a pentagon bounded by the lines $a(n) = 0$ for $n = 1, 2, 3, 5, 6$, which is not intersected by the line $a(4) = 0$, as seen in the following figures.
\begin{center}
\includegraphics[height=3in]{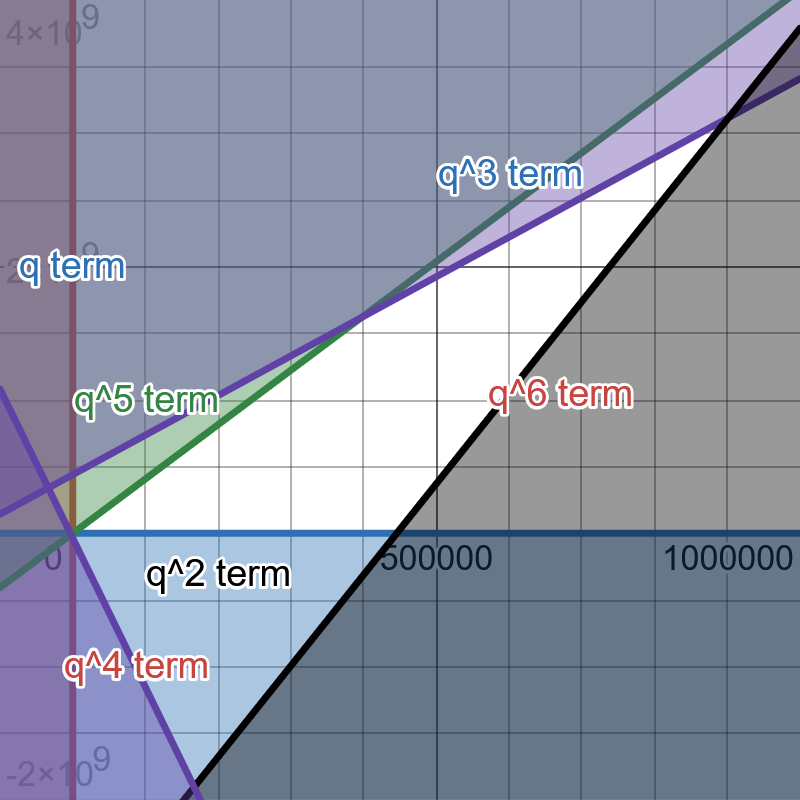} \includegraphics[height=3in]{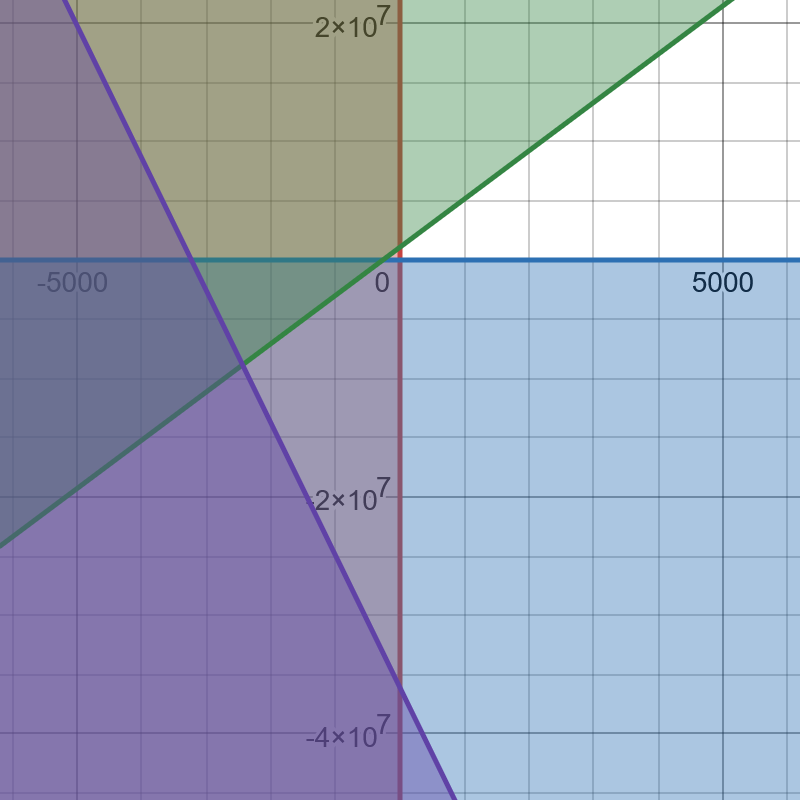}
\end{center}
By moving slightly outside each edge of the pentagon, computations allow us to obtain the forms
\small
\begin{align*}
F_{24, 0} - F_{24, 1} + F_{24, 2} = &1 - q + q^2 + 52220292 q^3 + 38995252952 q^4 + \cdots, \\
F_{24, 0} + F_{24, 1} - F_{24, 2} = &1 + q - q^2 + 52611708 q^3 + 39019411048 q^4 + \cdots, \\
F_{24, 0} + F_{24, 1} + 1096077 F_{24, 2} = &1 + q + 1096077 q^2 - 36 q^3 + 40203175288 q^4 + \cdots, \\
F_{24, 0} + 434170 F_{24, 1} + 1728600000 F_{24, 2} =
&1 + 434170 q + 1728600000 q^2 + 2029318200 q^3 \\ &+ 7150724505760 q^4 - 780499860 q^5 + \cdots, \\
F_{24, 0} + 445624 F_{24, 1} + F_{24, 2} = &1 + 445624 q + q^2 +  87243207792 q^3 + 5422202292952 q^4 \\ &+ 26508854569040 q^5 - 302988852 q^6 + \cdots,
\end{align*}
\normalsize
each of which has exactly one negative Fourier coefficient.

We now show that if the coefficients $a(1), a(2), \ldots, a(6)$ of $q^1, q^2, \ldots, q^6$ are nonnegative, then all other Fourier coefficients $a(n)$ of $f$ are nonnegative as well.  If we write $f$ in the basis that replaces $F_{24, 0}$ with $E_{24}$, so that
\[f = E_{24} + d_1 F_{24, 1} + d_2 F_{24, 2},\]
then computing the coefficients of $E_{24}$ gives $d_1 \approx A - 0.000554$ and $d_2 \approx B - 4650.635$.  Bounds on $A$ and $B$ from the pentagon in the $A$-$B$ plane above give $|d_1| \leq 901973$ and $|d_2| \leq 3117523826.926$.

Theorem~\ref{cuspformbound} shows that the $n$th Fourier coefficient of $d_1 F_{24, 1} + d_2 F_{24, 2}$ is bounded above by $1.7009796\cdot 10^{15} d(n) n^{23/2}$.  If the Eisenstein series coefficient is larger than the cusp form bound, then $a(n)$ will be positive, which happens when
\[\frac{-48}{B_{24}} \sigma_{23}(n) > 1.7009796\cdot 10^{15} d(n) n^{23/2}.\]
Direct computation shows that this is true for $n \geq 52$.

For $13 \leq n \leq 51$, the distance from the line $a(n) = 0$ to the origin of the $A$-$B$ plane is greater than the largest width of the pentagon, so none of these lines can intersect the pentagon, and if $a(1), \ldots, a(6)$ are positive then $a(n)$ must be as well.  Checking the lines $a(n) = 0 $ for $7 \leq n \leq 12$ directly confirms that none of these lines intersect the pentagon either.
Thus, if the coefficients $a(1), a(2), a(3), a(5)$, and $a(6)$ in the Fourier expansion of $f$ are nonnegative, then all Fourier coefficients of $f$ must be nonnegative.  We conclude that $A(24) = 6$.

Similar computations show that for any holomorphic modular form of weight $28$ with Fourier expansion $1 + O(q)$, if the coefficients of $q^1, q^2, q^4, q^7$, and $q^8$ are nonnegative then all other Fourier coefficients are as well.  Additionally, for each of the exponents $n = 1, 2, 4, 7, 8$ there exists a modular form of weight $28$ with all Fourier coefficients nonnegative except for the coefficient of $q^n$.  Thus, $A(28) = 8$.  Further, for any holomorphic form of weight $32$ with Fourier expansion $1 + O(q)$, if the coefficients of $q^1, q^2, q^5, q^9$, and $q^{10}$ are nonnegative then all other Fourier coefficients are as well.  For each of the exponents $n = 1, 2, 5, 9, 10$ there exists a modular form of weight 32 with all Fourier coefficients nonnegative except for the coefficient of $q^n$.  Therefore, $A(32) = 10$.

This approach to computing $A(k)$ becomes more difficult for larger weights.  Because the proof of the upper bound in Theorem~\ref{Akbounds} has the hypothesis that $k \geq 92$, we need an alternate method to compute $A(k)$ for $36 \leq k \leq 88$.  We describe one such algorithm here.

Suppose that the form
\[
  f = 1 + \sum_{n=1}^{\infty} a(n) q^{n} \in M_{k}
\]
has nonnegative Fourier coefficients, and write $f$ in terms of canonical basis elements as
\[
  f = F_{k,0} + \sum_{m=1}^{\ell} a(m) F_{k,m}.
\]
Note that for $1 \leq m \leq \ell$, the only term of this sum that can contribute to the Fourier coefficient of $q^m$ is the term coming from $F_{k, m}$; therefore, $a(m) \geq 0$ for $1 \leq m \leq \ell$.

We search for a positive integer $t$ for which $c_{m}(t) < 0$ for all $1 \leq m \leq \ell$. It is not clear whether such a value of $t$ must exist for all $k$, but for $12 \leq k \leq 88$, a minimal value of $t$ can be found by computing the Fourier expansions of the $F_{k,m}$.  For this minimal value of $t$, the assumption that the coefficient of $q^{t}$ in $f$ is nonnegative implies that for any $1 \leq m \leq \ell$, we have
\[
  0 \leq a(t) = c_{0}(t) + \sum_{m=1}^{\ell} a(m) c_{m}(t) \leq c_{0}(t) + a(m) c_{m}(t) \leq c_0(t)
\]
since all $a(m)$ are nonnegative and all $c_m(t)$ are negative.
This implies that $a(m) \leq \frac{-c_{0}(t)}{c_{m}(t)}$. If we represent $f$ as the sum
\[
  f = E_{k} + \sum_{n=1}^{\infty} c(n) q^{n},
\]
of an Eisenstein series and a cusp form, it follows that
\[
  |c(m)| \leq \max \left\{ -\frac{2k}{B_{k}} \sigma_{k-1}(m),
  \frac{-c_{0}(t)}{c_{m}(t)} + \frac{2k}{B_{k}} \sigma_{k-1}(m) \right\}.
\]
We insert this bound into Theorem~\ref{cuspformbound} and, for each weight $k$ with $12\leq k \leq 88$, we compute a constant $C_{2}$
so that
\begin{equation}\label{C2bound}
  -\frac{2k}{B_{k}} \sigma_{k-1}(n) - C_{2} d(n) n^{\frac{k-1}{2}}
  \leq a(n) \leq -\frac{2k}{B_{k}} \sigma_{k-1}(n) + C_{2} d(n) n^{\frac{k-1}{2}}.
\end{equation}
The upper bound in this inequality will be used to prove
Theorem~\ref{upperbound} when $k$ is small.  Using the lower bound and the fact that $d(n) \leq 2 \sqrt{n}$, we find that the $n$th coefficient of $f$ is positive provided that
\[
  n > \left(-\frac{C_{2} B_{k}}{k}\right)^{\frac{1}{k/2-1}}.
\]
It follows that $A(k)$ has an upper bound $B(k)$ given by
\[
  A(k) \leq \max \left\{ \ell, t, \left\lfloor \left(-\frac{C_{2} B_{k}}{k}\right)^{\frac{1}{k/2-1}} \right\rfloor + 1 \right\} = B(k).
\]

Expressing $f$ in terms of canonical basis elements as above shows that the coefficient of $q^n$ in $f$ must be equal to \[c_0(n) + \sum_{m=1}^\ell a(m) c_m(n) \geq 0.\]  Every normalized form of weight $k$ with no nonnegative Fourier coefficients must satisfy this sequence of linear inequalities on $a(1), \ldots, a(m)$ for each $n=1, 2, \ldots, B(k)$.

For each weight $36 \leq k \leq 88$, we use the \texttt{linarith} tactic in Lean~\cite{Lean} to find the smallest $N$ for which assuming that all of the linear inequalities for $n=1, 2, \ldots, N$ are true implies that all of the linear inequalities for $n = N+1, N+2, \ldots, B(k)$ must also hold.  Thus, if the first $N$ Fourier coefficients of $f$ are nonnegative, then all other Fourier coefficients of $f$ are nonnegative (or positive, once $n > B(k)$).  Because \texttt{linarith} is a decision procedure and the first $N-1$ inequalities do not imply inequality $N$, there must be some set of values $a(1), a(2), \ldots, a(\ell)$ for which the $N$th Fourier coefficient of $f$ is negative and the first $N-1$ coefficients are nonnegative, and we conclude that $N = A(k)$.

Table~\ref{table2} gives the results of the calculations described above for $12 \leq k \leq 88$, as well as the actual values of $A(k)$ for each weight.

\begin{table}
\begin{center}
\caption{Values of $A(k)$ and related calculations for $12 \leq k \leq 88$}
\label{table2}
\begin{tabular}{ccccc|ccccc}
$k$ & $t$ & $C_{2}$ & $B(k)$ & $A(k)$ & $k$ & $t$ & $C_{2}$ & $B(k)$ & $A(k)$\\
\hline
$12$ & $2$ &  $6.89 \cdot 10^{12}$ & $171$ & $2$ & $52$ & $24$ & $1.19 \cdot 10^{21}$ & $70$ & $22$\\
$16$ & $3$ &  $2.40 \cdot 10^{13}$ & $73$ & $3$ &  $56$ & $28$ & $4.69 \cdot 10^{22}$ & $80$ & $26$\\
$20$ & $4$ &  $6.79 \cdot 10^{13}$ & $50$ & $4$ &  $60$ & $51$ & $1.92 \cdot 10^{30}$ & $147$ & $29$\\
$24$ & $7$ &  $6.37 \cdot 10^{16}$ & $71$ & $6$ &  $64$ & $74$ & $3.23 \cdot 10^{33}$ & $182$ & $33$\\
$28$ & $8$ &  $1.58 \cdot 10^{17}$ & $61$ & $8$ &  $68$ & $36$ & $6.00 \cdot 10^{23}$ & $90$ & $36$\\
$32$ & $10$ & $4.66 \cdot 10^{16}$ & $49$ & $10$ & $72$ & $56$ & $3.70 \cdot 10^{28}$ & $125$ & $41$\\
$36$ & $12$ & $2.37 \cdot 10^{17}$ & $51$ & $12$ & $76$ & $48$ & $1.41 \cdot 10^{26}$ & $108$ & $45$\\
$40$ & $14$ & $4.82 \cdot 10^{17}$ & $51$ & $14$ & $80$ & $91$ & $4.59 \cdot 10^{34}$ & $182$ & $49$\\
$44$ & $16$ & $9.90 \cdot 10^{17}$ & $52$ & $16$ & $84$ & $84$ & $1.03 \cdot 10^{33}$ & $165$ & $54$\\
$48$ & $22$ & $8.88 \cdot 10^{20}$ & $70$ & $19$ & $88$ & $60$ & $1.72 \cdot 10^{26}$ & $116$  & $59$\\
\end{tabular}
\end{center}
\end{table}

\section{Well-definedness of $A(k)$}\label{WellDefined}

We now study the geometry of the set of modular forms of weight $k$ with only nonnegative coefficients, and prove that $A(k)$ is well-defined.  Therefore, for a form $f(z) = 1 + \sum a(n) q^{n} \in M_{k}$, checking that $a(n) \geq 0$ for all $n$ up to some finite bound guarantees that all $a(n)$ are nonnegative.

Fix a positive integer $k \equiv 0 \pmod{4}$.  As noted above, $M_k$ has dimension $\ell+1$, and any normalized modular form in $M_k$ with nonnegative Fourier coefficients is determined by the coefficients of $q^1, q^2, \ldots, q^\ell$. Let $S \subseteq \R^\ell$ be the set of $\ell$-tuples $(a(1), a(2), \ldots, a(\ell))$ where the modular form $1 + a(1) q + \cdots + a(\ell) q^{\ell} + \cdots \in M_{k}$ has no negative Fourier coefficients.

If $F_{k,m} = \sum c_{m}(n) q^{n}$ is the unique modular form in $M_{k}$ whose Fourier expansion begins $q^{m} + O(q^
{\ell+1})$, then $(a(1), a(2), \ldots, a(\ell)) \in S$ corresponds to a modular form
\[
  F_{k,0} + \sum_{m=1}^{\ell} a(m) F_{k,m}.
\]
The requirement that the $n$th coefficient is nonnegative becomes
\[
  c_{0}(n) + \sum_{m=1}^{\ell} a(m) c_{m}(n) \geq 0.
\]
This is a closed half-space in $\R^{\ell}$. Call this closed half-space $S_{n}$, so that $S = \bigcap_{n=1}^{\infty} S_{n}$. Because $S$ is an intersection of closed sets, it is a closed set.

We next define a related sequence of sets $T_n$, where
\[
T_{n} = \{ (a(1), a(2), \ldots, a(\ell)) : \text{ the coefficient of } q^{n} \text{ in }
\sum_{m=1}^{\ell} a(m) F_{k,m} \text{ is nonnegative} \}.
\]
Each $T_{n}$ is a half-space in $\R^{\ell}$ whose boundary is parallel to that of $S_{n}$, since its defining condition can be written
\[\sum_{m=1}^{\ell} a(m) c_{m}(n) \geq 0.\]
Note that the boundary of $T_n$ goes through the origin.

We say that a set $S \subseteq \R^{\ell}$ is a convex polytope if it is the intersection of a set of closed half-spaces. A polytope is bounded if it is contained in a ball of finite radius, and a polytope is finite if it is the intersection of finitely many closed half-spaces.

With these definitions, we have the following theorem.
\begin{thm}
The set $S$ is a bounded finite convex polytope in $\R^{\ell}$. In particular, there is a finite set of positive integers $A$ so that
\[
  S = \bigcap_{n \in A} S_{n}.
\]
It follows that all of the coefficients of a modular form $f \in M_{k}$ are nonnegative if and only if
the coefficient of $q^{n}$ is nonnegative for all $n \in A$. Consequently, $A(k) \leq \max A$.
\end{thm}
\begin{proof}
First, we show that $S$ is a convex polytope. Next we show that $S$ is bounded. Finally we show that $S$ is a finite polytope by showing that the distance between the boundary of $S_{n}$ and the origin tends to infinity with $n$.

The fact that $S$ is convex follows from the fact that a nonnegative
linear combination of forms with nonnegative coefficients also has nonnegative coefficients. To see that $S$ is bounded, we show that there is a finite set of integers $n_{1}$, $n_{2}$, $\ldots$, $n_{s}$ so that
\[
  S \subseteq \bigcap_{i=1}^{s} S_{n_{i}}.
\]
We use the well-known fact that a nonzero cusp form must have a negative Fourier coefficient.
As a consequence, the intersection of the $T_n$ is the single point at the origin:
\[
  \bigcap_{n=1}^{\infty} T_{n} = \{ (0,0,0,\ldots, 0)\}.
\]
Hence, the unit sphere in $\R^{\ell}$ satisfies $\mathbb{S}^{\ell-1} \subseteq \bigcup_{n=1}^{\infty} T_{n}^{c}$.
Note that $\mathbb{S}^{\ell-1}$ is compact, and $\bigcup_{n=1}^{\infty} T_{n}^{c}$ is an open cover for it.
Thus, there is a finite subset $B \subseteq \{ 1, 2, \ldots \}$ such that
\[
  \mathbb{S}^{\ell-1} \subseteq \bigcup_{n \in B} T_{n}^{c}.
\]
This implies that
\[
  \bigcap_{n \in B} T_{n} = \{ (0,0,0,\ldots, 0)\}.
\]
Indeed, if $\vec{v}$ is a nonzero vector in $\bigcap_{n \in B} T_{n}$,
then $\frac{1}{\|\vec{v}\|} \vec{v}$ must also be in $\bigcap_{n \in B} T_{n} \cap \mathbb{S}^{\ell-1} = \emptyset$.

It follows that every direction away from the origin is blocked by one of the boundaries of $T_{n}$ for $n \in B$, which implies that $\bigcap_{n \in B} S_{n}$ is bounded. This follows because every unbounded, closed, convex subset of $\R^{\ell}$ must contain a ray, and if a ray were contained in $\bigcap_{n \in B} S_{n}$, then a translation of this ray to the origin would be contained in $\bigcap_{n \in B} T_{n}$.
Since $\bigcap_{n \in B} S_{n}$ is bounded and
$S \subseteq \bigcap_{n \in B} S_{n}$, it follows that $S$ is bounded.

This is not sufficient to show that $S = \bigcap_{n \in B} S_{n}$, because
a bounded convex polytope need not be finite. However, we will show that there is some finite set $A$ with $B \subseteq A$ so that $S = \bigcap_{n \in A} S_{n}$. To show this, we will show
that the distance between the origin in $\R^{\ell}$ and the closest point to the origin that is on the boundary of
$S_{n}$ tends to infinity with $n$. This implies that there is a positive integer $N$ so that
if $n \geq N$, then $\bigcap_{n \in B} S_{n} \subseteq S_{N}$. Hence
\[
  S = \bigcap_{n=1}^{\infty} S_{n} = \bigcap_{n=1}^{N} S_{n}
\]
and so taking $A = \{ 1, 2, \ldots, N \}$ suffices.

The equation for the boundary of $S_{n}$, in terms of the variables $a(1)$, $\ldots$, $a(\ell)$, is
\[
  \sum_{m=1}^{\ell} a(m) c_{m}(n) = -c_{0}(n).
\]
In general, the distance between the origin and the hyperplane $\vec{v} \cdot \vec{x} = c$ is
$\frac{|c|}{\sqrt{\vec{v} \cdot \vec{v}}}$. Thus, it suffices to show that
\[
  \frac{|c_{0}(n)|}{\sqrt{ c_{1}(n)^{2} + c_{2}(n)^{2} + \cdots + c_{\ell}(n)^{2}}} \to \infty
\]
as $n \to \infty$.

For each $r$ with $1 \leq r \leq n$, $c_{r}(n)$ is the $n$th coefficient of a cusp form of weight $k$
and thus, $|c_{r}(n)| \ll_{r} d(n) n^{\frac{k-1}{2}}$. On the other hand, $c_{0}(n)$ is the $n$th coefficient
of a modular form which is $E_{k}$ plus a cusp form, and so $|c_{0}(n)| \gg n^{k-1}$. Combining these bounds shows that
\[
  \frac{|c_{0}(n)|}{\sqrt{ c_{1}(n)^{2} + c_{2}(n)^{2} + \cdots + c_{\ell}(n)^{2}}} \gg \frac{n^{k-1}}{d(n) n^{\frac{k-1}{2}}} \to \infty
\]
since $d(n) \ll n^{\epsilon}$ for all $\epsilon > 0$.
\end{proof}

\section{Lower bound on $A(k)$}
\label{lowerbound}
In this section, we prove the lower bound for $A(k)$ in Theorem~\ref{Akbounds} by exhibiting a modular form of weight $k$ whose Fourier expansion has many positive Fourier coefficients before its first negative coefficient.  Assume that $k \geq 16$, since $\frac{(k-1)^{2}}{16 \pi^{2}} < 1$ if $k \leq 12$.

If $P_{1} = \sum_{n=1}^{\infty} b(n) q^{n}$ is the Poincar\'e series of index $1$, then equation \eqref{poincare} gives
\[
  b(n) = n^{\frac{k-1}{2}} \left[\delta_{n,1} + 2 \pi \sum_{c=1}^{\infty} \frac{K(1,n;c)}{c} J_{k-1}\left(\frac{4 \pi \sqrt{n}}{c}\right)\right].
\]
We will show that when $n \leq \frac{(k-1)^{2}}{16 \pi^{2}}$, which is equivalent to $4 \pi \sqrt{n} \leq k-1$, the $c = 1$ term of the infinite sum is positive and is larger than the sum of the absolute values of the terms for $c \geq 2$, so $b(n)$ must be positive.  It follows that the first $\lfloor \frac{(k-1)^{2}}{16 \pi^{2}} \rfloor$ coefficients of the cusp form $P_1$ are positive. Since every cusp form has negative coefficients, the first negative Fourier coefficient of $P_1$ must be the coefficient of $q^n$ for some $n > \frac{(k-1)^{2}}{16 \pi^{2}}$, meaning that $\frac{(k-1)^{2}}{16 \pi^{2}} < A(k)$. The argument we give is substantially similar
to Theorem 2.1 of Noam Kimmel's preprint \cite{KimmelPreprint}. In this result, Kimmel
shows that the coefficient of $q^{1}$ in the level $1$ Poincar\'e series $P_{m}$ is nonzero.
Because of a simple relationship between the coefficient of $q^{1}$ in $P_{m}$
and the coefficient of $q^{m}$ in $P_{1}$, our result is equivalent to Kimmel's.

We compute that $K(1,n;1) = 1$, so the $c=1$ term is equal to $J_{k-1}(4\pi\sqrt{n})$.  The inequality
\[
  \frac{J_{\nu}(\nu x)}{x^{\nu} J_{\nu}(\nu)} \geq 1
\]
is valid for $\nu \geq 0$ and $0 < x \leq 1$ (item 10.14.7 of \cite{NIST:DLMF}); applying it with $\nu = k-1$ and $x = \frac{4\pi \sqrt{n}}{k-1}$ gives that
\[
  J_{k-1}(4 \pi \sqrt{n}) \geq \left(\frac{4 \pi \sqrt{n}}{k-1}\right)^{k-1} J_{k-1}(k-1).
\]
Watson shows (page 260 of \cite{Watson}) that $\nu^{1/3} J_{\nu}(\nu)$ is monotonically increasing, and that \[\lim_{\nu \to \infty} \nu^{1/3} J_{\nu}(\nu) = \frac{\Gamma(1/3)}{2^{2/3} 3^{1/6} \pi}.\] It follows that for $k \geq 16$, we have
\[
  J_{k-1}(k-1) \geq \frac{0.447}{(k-1)^{1/3}},
\]
and a lower bound for the $c=1$ term is given by
\[
J_{k-1}(4 \pi \sqrt{n}) \geq \frac{0.447}{(k-1)^{1/3}} \left(\frac{4 \pi \sqrt{n}}{k-1}\right)^{k-1}.
\]

For $c \geq 2$, the upper bound
\[
  |J_{\nu}(z)| \leq \frac{\left|\frac{1}{2} z\right|^{\nu}}{\Gamma(\nu+1)}
\]
is valid for $z$ real and $\nu \geq -1/2$ (item 10.14.4 of \cite{NIST:DLMF}).  Applying it with $\nu = k-1$ and $z = \frac{4\pi \sqrt{n}}{c}$ gives
\[
  \left|J_{k-1}\left(\frac{4 \pi \sqrt{n}}{c}\right)\right| \leq
  \left(\frac{2 \pi \sqrt{n}}{c}\right)^{k-1} \cdot \frac{1}{\Gamma(k)}
  = \frac{(2 \pi \sqrt{n})^{k-1}}{(k-1)!} \cdot \frac{1}{c^{k-1}}.
\]
From the inequality \eqref{facbound}, we obtain
\[
  \left|J_{k-1}\left(\frac{4 \pi \sqrt{n}}{c}\right)\right| \leq \left(\frac{2 \pi e \sqrt{n}}{k-1}\right)^{k-1} \cdot \frac{1}{\sqrt{2 \pi (k-1)}} \cdot \frac{1}{c^{k-1}}.
\]
The Weil bound for the Kloosterman sum (see equation 1.60 of \cite[p. 19]{IK}) states that \[|K(m,n;c)| \leq d(c) \gcd(m,n,c)^{1/2} \sqrt{c}.\] This gives
\begin{align*}
  \sum_{c=2}^{\infty} \left|\frac{K(1,n;c)}{c} J_{k-1}\left(\frac{4 \pi \sqrt{n}}{c}\right)\right| &\leq \left(\frac{2 \pi e \sqrt{n}}{k-1}\right)^{k-1} \cdot \frac{1}{\sqrt{ 2 \pi (k-1)}}
  \sum_{c=2}^{\infty} \frac{d(c)}{c^{k-1/2}}\\
  &\leq \frac{0.447}{(k-1)^{1/3}} \cdot \left(\frac{4 \pi \sqrt{n}}{k-1}\right)^{k-1}
  \cdot \left(\frac{0.893}{(k-1)^{1/6}} \cdot \left(\frac{e}{2}\right)^{k-1} \cdot \sum_{c=2}^{\infty} \frac{d(c)}{c^{k-1/2}}\right).
\end{align*}
To bound the sum on $c$ we use the fact that $\zeta(s)^{2} = \sum_{n=1}^{\infty} \frac{d(n)}{n^{s}}$.
For $s > 1$ real we have
\begin{align*}
  \zeta(s) &= 1 + \frac{1}{2^{s}} + \sum_{n=3}^{\infty} \frac{1}{n^{s}}
  = 1 + \frac{1}{2^{s}} + \int_{3}^{\infty} \frac{1}{\lfloor x \rfloor^{s}} \, dx\\
  &\leq 1 + \frac{1}{2^{s}} + \int_{3}^{\infty} \frac{1}{(x-1)^{s}} \, dx = 1 + \frac{1}{2^{s}} + \int_{2}^{\infty} \frac{1}{x^{s}} \, dx = 1 + \frac{s+1}{2^{s} (s-1)}.
\end{align*}
This gives that
\begin{align*}
  \sum_{c=2}^{\infty} \frac{d(c)}{c^{k-1/2}} = \zeta(k-1/2)^{2} - 1
  &\leq \left(1 + \frac{k+1/2}{2^{k-1/2} (k-1/2)}\right)^{2} - 1\\
  &= \frac{2k+1}{2^{k-1/2} (k-1/2)} + \frac{(k+1/2)^{2}}{2^{2k-1} (k-1/2)^{2}}\\
  &= \frac{1}{2^{k-1}} \left(\sqrt{2} \cdot \frac{2k+1}{k-1/2} + \frac{1}{2^{k}} \cdot \frac{(k+1/2)^{2}}{(k-1/2)^{2}}\right).
\end{align*}

It is straightforward to see that the quantity multiplied by $\frac{1}{2^{k-1}}$ above is decreasing, so letting $k=16$ gives
\[ \sum_{c=2}^{\infty} \frac{d(c)}{c^{k-1/2}} \leq \frac{3.02}{2^{k-1}}. \]
Therefore, the sum of the absolute values of the terms in the sum with $c \geq 2$ is bounded above by
\[
  \frac{0.447}{(k-1)^{1/3}} \cdot \left(\frac{4 \pi \sqrt{n}}{k-1}\right)^{k-1}
  \cdot \left(\frac{2.697}{(k-1)^{1/6}} (e/4)^{k-1}\right).
\]
Computation shows that $\frac{2.697 (e/4)^{k-1}}{(k-1)^{1/6}} < 1$ for $k \geq 16$, so the $c=1$ term must be larger than the sum of the absolute values of all other terms.  This concludes the proof of the lower bound in Theorem~\ref{Akbounds}.

We remark that the main term in the formula for the $n$th coefficient
is $J_{k-1}(4 \pi \sqrt{n})$. According to item 10.21.40 in \cite{NIST:DLMF}, the first zero of $J_{\nu}(x)$ is asymptotically about $\nu + 1.85 \nu^{1/3}$
and so $J_{k-1}(4 \pi \sqrt{n})$ is zero when $4 \pi \sqrt{n} \approx (k-1) + 1.85 (k-1)^{1/3}$.
This suggests that there is very little room for improvement in the above argument.

\section{Upper bound on $A(k)$}
\label{finalsec}
In this section, we prove three preliminary results and use these to give an upper bound on $A(k)$ with $k \geq 92$.  Combining this result with the results in Section~\ref{smallk} allows us to then prove Theorems~\ref{Akbounds},~\ref{positivity}, and~\ref{upperbound}.

Recall that $\ell = \lfloor \frac{k}{12}\rfloor$. We begin with a lemma bounding the contribution from the cusp form part of a modular form to its Fourier coefficients when the first $\ell+1$ Fourier coefficients are nonnegative.

\begin{lemma}
\label{smallC}
Assume that $k \geq 64$ and $f(z) = 1 + \sum_{n=1}^{\infty} a(n) q^{n} \in M_{k}$. Assume that $a(n) \geq 0$ for $1 \leq n \leq \ell$
and let $C = \sum_{n=1}^{\ell} a(n)$. Write $f(z) = E_{k}(z) + \sum_{n=1}^{\infty} c(n) q^{n}$. Then
\[
  |c(n)| \leq 12 (C + 1.5^{k}) \sqrt{\log k} \, d(n) n^{\frac{k-1}{2}}.
\]
\end{lemma}
\begin{proof}
First, we have that
\begin{align*}
   \sum_{m=1}^{\ell} |c(m)| &\leq \sum_{m=1}^{\ell} a(m) + \left(-\frac{2k}{B_{k}} \right) \sum_{m=1}^{\ell} \sigma_{k-1}(m)\\
  &= C + \left(-\frac{2k}{B_{k}}\right) \sum_{m=1}^{\ell} \sum_{d | m} \left(\frac{m}{d}\right)^{k-1} \leq C + \left(-\frac{2k}{B_{k}}\right) \sum_{m=1}^{\ell} m^{k-1} \sum_{d=1}^{\infty} \frac{1}{d^{k-1}}\\
  &= C + \left(-\frac{2k}{B_{k}}\right) \sum_{m=1}^{\ell} m^{k-1} \zeta(k-1) \leq C + \frac{(2 \pi)^{k} \zeta(k-1)}{(k-1)! \zeta(k)} \ell \cdot \ell^{k-1}\\
  &\leq C + \frac{(2 \pi \ell)^{k} k \zeta(k-1)}{k! \zeta(k)}.
\end{align*}
We use the lower bound for $k!$ in equation~(\ref{facbound}) above, the fact that $\zeta(k)$ is decreasing and greater than $1$ for real $k>1$, and the fact that $\ell = \lfloor \frac{k}{12} \rfloor$ for $k \equiv 0 \pmod{4}$ to obtain
\[
  \sum_{m=1}^{\ell} |c(m)| \leq C + \frac{(2 \pi (k/12))^{k} k \zeta(63)}{\sqrt{2 \pi k} (k/e)^{k}}
  \leq C + \frac{\zeta(63)\sqrt{k}}{\sqrt{2 \pi}} \left(\frac{2 \pi e}{12}\right)^{k}.
\]
It is straightforward to verify that
$\zeta(63)\sqrt{\frac{k}{2\pi}} \left(\frac{2 \pi e}{12}\right)^{k} < 1.5^{k}$ for $k \geq 64$ and
hence $\sum_{m=1}^{\ell} |c(m)| \leq C + 1.5^{k}$.

We will combine this bound with Theorem~\ref{cuspformbound}, which states that
\small
\begin{equation}\label{JRbound}
  |c(n)| \leq
  \sqrt{\log k} \left(11 \sqrt{\sum_{m=1}^{\ell} \frac{|c(m)|^{2}}{m^{k-1}}} + \frac{e^{18.74} (41.41)^{k/2}}{k^{(k-1)/2}} \sum_{m=1}^{\ell} |c(m)| e^{-7.288m}\right) d(n) n^{\frac{k-1}{2}}.
\end{equation}
\normalsize
We repeatedly use the fact that if $x_{1}$, $\ldots$, $x_{\ell}$ are nonnegative real numbers, then
\[
  \sqrt{\sum_{m=1}^{\ell} x_{n}^{2}} \leq \sum_{m=1}^{\ell} x_{n}.
\]
To bound the first term in the parentheses in
\eqref{JRbound}, we use the Cauchy-Schwarz inequality to conclude that
\begin{align*}
  \sqrt{\sum_{m=1}^{\ell} \frac{|c(m)|^{2}}{m^{k-1}}} &\leq \sum_{m=1}^{\ell} \frac{|c(m)|}{m^{(k-1)/2}} \leq \left(\sum_{m=1}^{\ell} |c(m)|^{2}\right)^{1/2} \cdot \left(\sum_{m=1}^{\ell} \frac{1}{m^{k-1}}\right)^{1/2}\\
  &\leq \left(\sum_{m=1}^{\ell} |c(m)|\right) \cdot \zeta(k-1)^{1/2} \leq (C + 1.5^{k}) \zeta(63)^{1/2} \leq (1+6 \cdot 10^{-20}) (C + 1.5^{k}).
\end{align*}
To bound the second term in parentheses in \eqref{JRbound}, we again use the Cauchy-Schwarz inequality to obtain
\begin{align*}
    \sum_{m=1}^{\ell} |c(m)| e^{-7.288m}
    &\leq \left(\sum_{m=1}^{\ell} |c(m)|^{2}\right)^{1/2} \left(\sum_{m=1}^{\ell} e^{-14.576m}\right)^{1/2}\\
    &\leq \sum_{m=1}^{\ell} |c(m)| e^{-7.288}
    \left(\sum_{t=0}^{\infty} e^{-14.576 \cdot t}\right)^{1/2}\\
    &\leq \sum_{m=1}^{\ell} |c(m)| e^{-7.288}
    \cdot (1 + 2.4 \cdot 10^{-7})\\
    &\leq e^{-7.287} (C + 1.5^{k}).
\end{align*}
Returning to \eqref{JRbound} we obtain
\[
  |c(n)| \leq \left(11 \cdot (1 + 6 \cdot 10^{-20}) \sqrt{\log k}
  + e^{11.433} \sqrt{41.41} \left(\frac{41.41}{k}\right)^{(k-1)/2}\right) (C + 1.5^{k}) d(n) n^{\frac{k-1}{2}}.
\]
Dividing the first factor by $\sqrt{\log k}$, we obtain a function of $k$ which is clearly decreasing once $k \geq 41.41$. For $k \geq 64$, this quantity is less than $11.323 < 12$.
\end{proof}

We will also use the following lemma to obtain upper bounds on the tails of infinite sums that arise when bounding values of modular forms at specific points.
\begin{lemma}
\label{lemma3variation}
Let $s$ be a positive integer and $\alpha, \beta \in \R$ with $\alpha > 0$ and $\beta < 0$. If $s \beta < -\alpha$, then
\[
  \sum_{n=s}^{\infty} n^{\alpha} e^{n \beta} \leq s^{\alpha} e^{s \beta} \cdot \frac{1}{1 - e^{\alpha/s + \beta}}.
\]
\end{lemma}
\begin{proof}
Letting $f(x) = \alpha \log(x) + x \beta$, the sum is $\sum_{n=s}^{\infty} e^{f(n)}$. For $x \geq s$, we have that
$f'(x) = \frac{\alpha}{x} + \beta \leq \frac{\alpha}{s} + \beta$. Thus,
$f(x) \leq f(s) + \int_{s}^{x} f'(t) \, dt \leq f(s) + (x-s) \left(\frac{\alpha}{s} + \beta\right)$ and so
\[
  e^{f(n)} \leq e^{f(s)} \cdot e^{(n-s) \cdot (\alpha/s + \beta)}.
\]
Therefore,
\[
  \sum_{n=s}^{\infty} e^{f(n)} \leq e^{f(s)} \cdot \sum_{n=s}^{\infty} e^{(n-s) (\alpha/s + \beta)} = e^{f(s)} \sum_{n=0}^\infty e^{n(\alpha/s + \beta)}.
\]
The infinite sum on the right hand side above is a geometric series. Summing this series gives the desired result.
\end{proof}

The next proposition shows that if the first $\ell$ coefficients of a modular form are nonnegative and their sum $C$ is too large, then the modular form must have a negative Fourier coefficient in a particular range.
\begin{prop}
\label{boundonC}
Suppose that $k \geq 84$, $f(z) = 1 + \sum_{n=1}^{\infty} a(n) q^{n} \in M_{k}$,
and $a(n) \geq 0$ for $1 \leq n \leq \ell$. Let $C = \sum_{n=1}^{\ell} a(n)$.
If $C > 1.806 (k/2 \pi)^{k} (\log k + \log \log k)^{k}$, then there is a positive integer $n \in [\ell+1,k^{2} (\log k + \log \log k)^{2}]$ with $a(n) < 0$.
\end{prop}
\begin{proof}
Applying the transformation law $f(-1/z) = z^{k} f(z)$ and setting $z = iy$ gives the equation
\[
  1 + \sum_{n=1}^{\infty} a(n) e^{-2 \pi n/y} = y^{k} \left(1 + \sum_{n=1}^{\infty} a(n) e^{-2 \pi n y}\right).
\]
We rewrite this as
\begin{equation}
\label{keyeq}
  1 + \sum_{n=1}^{\infty} a(n) \left(e^{-2 \pi n/y} - y^{k} e^{-2 \pi n y}\right) = y^{k}.
\end{equation}
The quantity $e^{-2 \pi n/y} - y^{k} e^{-2 \pi n y}$ is positive if and only if
$\frac{2 \pi (y-1/y)}{\log(y)} > k/n$. In particular, if we choose $y$ sufficiently large that $\frac{2 \pi (y-1/y)}{\log(y)} > k$, or equivalently that $e^{-2 \pi/y} - y^{k} e^{-2 \pi y} > 0$, it follows that
$e^{-2 \pi n / y} - y^{k} e^{-2 \pi n y} > 0$ for all positive integers $n$.

We now show that the choice $y = \frac{k}{2 \pi} (\log k + \log \log k)$ has this property. We compute that
\begin{align*}
    y^{k} &= \left(\frac{k}{2 \pi}\right)^{k} (\log k + \log \log k)^{k}\\
    &= \frac{k^{k} (\log k)^{k}}{(2 \pi)^{k}} \left(1 + \frac{\log \log k}{\log k}\right)^{k},
\end{align*}
and
\[
  e^{-2 \pi y} = e^{-k (\log k + \log \log k)} = \frac{1}{k^{k}} \cdot \frac{1}{(\log k)^{k}}.
\]
In particular, $y^{k} e^{-2 \pi y} = \frac{1}{(2 \pi)^{k}} \left(1 + \frac{\log \log k}{\log k}\right)^{k}
\leq \frac{1}{(2 \pi)^{k}} 1.336^{k} \leq (0.213)^{k}$ for $k \geq 84$. On the other hand, for $1 \leq n \leq \ell$, we have
\[
  e^{-2 \pi n/y} \geq e^{-2 \pi \ell/y} \geq e^{-2 \pi (k/12)/y} = e^{-\frac{4 \pi^{2}}{12 (\log k + \log \log k)}} \geq e^{-0.59}
\]
for $k \geq 84$. It follows that
\[e^{-2 \pi/y} - y^{k} e^{-2 \pi y} \geq e^{-0.59}-(0.213)^k > 0\]
for $k \geq 84$, as required.  Additionally, we compute that
\[
  \sum_{n=1}^{\ell} a(n) (e^{-2 \pi n/y} - y^{k} e^{-2 \pi n y}) \geq \sum_{n=1}^{\ell} a(n) (e^{-0.59} - (0.213)^k)  \geq \sum_{n=1}^{\ell} a(n) (0.554327 - 3.84\cdot 10^{-57}) \geq 0.554C.
\]
It follows from~\eqref{keyeq} that for this value of $y$, we have
\[
  \sum_{n=\ell+1}^{\infty} a(n) \left(e^{-2 \pi n/y} - y^{k} e^{-2 \pi n y}\right) \leq \left(\frac{k}{2 \pi}\right)^{k} (\log k + \log \log k)^{k} - 0.554C - 1.
\]
The assumption that $C > 1.806 \left(\frac{k}{2 \pi}\right)^{k} (\log k + \log \log k)^{k}$ allows us to obtain an upper bound for the series of
\begin{equation}\label{TailBound}
\sum_{n=\ell+1}^{\infty} a(n) \left(e^{-2 \pi n/y} - y^{k} e^{-2 \pi n y}\right) < \frac{1}{1.806} C - 0.554C < -0.00029C.
\end{equation}

If we can find an integer $s$ for which
\[
  \sum_{n=s}^{\infty} |a(n)| \left(e^{-2 \pi n /y} - y^{k} e^{-2 \pi n y}\right)  \leq \sum_{n=s}^{\infty} |a(n)| e^{-2 \pi n/ y} < 0.00029C,
\]
then the bound~(\ref{TailBound}) on the sum forces $\sum_{n=\ell+1}^{s} a(n) \left(e^{-2 \pi n /y} - y^{k} e^{-2 \pi n y}\right)$ to be negative,
and since $e^{-2 \pi n / y} - y^{k} e^{-2 \pi n y} > 0$ for all positive integers $n$, there must be at least one integer $n$ in the interval $[\ell+1,s]$ with $a(n) < 0$.
We will show that the choice $s = \lceil k^{2} (\log k + \log \log k)^{2} \rceil$ satisfies this property.

Using Lemma~\ref{smallC} to bound the cusp form part of $a(n)$, we know that
\[
  |a(n)| \leq -\frac{2k}{B_{k}} \sigma_{k-1}(n) + 12 \sqrt{\log k} (C + 1.5^{k}) d(n) n^{(k-1)/2}.
\]
We begin by bounding
\[
  \sum_{n=s}^{\infty} \left(-\frac{2k}{B_{k}}\right) \sigma_{k-1}(n) e^{-2 \pi n /y}.
\]
Because $k \geq 84$ and the $\zeta$-function is decreasing on the interval $[83, \infty)$, we have $\sigma_{k-1}(n) \leq \zeta(k-1) n^{k-1} \leq (1 + 2 \cdot 10^{-25}) n^{k-1}$.  We use the inequality~(\ref{facbound}) to see that
\[
  \left(-\frac{2k}{B_{k}}\right) = \frac{(2 \pi)^{k} \cdot k}{k! \zeta(k)}
  \leq \frac{(2 \pi)^{k} \cdot k}{\sqrt{2 \pi k} (k/e)^{k}}
  = \frac{\sqrt{k}}{\sqrt{2 \pi}} \left(\frac{2 \pi e}{k}\right)^{k}.
\]
We apply Lemma~\ref{lemma3variation} with $\alpha = k-1$ and $\beta = -2 \pi/y$, first checking that
\[s \beta = \frac{-2 \pi \lceil k^{2} (\log k + \log \log k)^{2} \rceil}{\frac{k}{2 \pi} (\log k + \log \log k)} <
-4\pi^2 k (\log k + \log \log k) < 1-k = -\alpha,\]
and obtain the bound
\[
  \zeta(k-1)  \left(-\frac{2k}{B_{k}}\right) \sum_{n=s}^{\infty}  n^{k-1} e^{-2 \pi n / y} \leq (1+2 \cdot 10^{-25}) \cdot \frac{\sqrt{k}}{s \sqrt{2 \pi}} \left(\frac{2 \pi e s}{k}\right)^{k} \frac{e^{-2 \pi s/y}}{1 - e^{(k-1)/s - 2 \pi/y}}.
\]
The function $(k-1)\log(s) -2\pi s/y$ is decreasing as a function of $s$, which can be seen by computing its derivative $\frac{k-1}{s} - 2 \pi/y = \frac{\alpha}{s} + \beta$, which is negative since $s\beta < -\alpha$.  Thus, the function $s^{k-1} e^{-2 \pi s / y}$ is decreasing as a function of $s$, and we can obtain an upper bound for it by replacing our choice of $s$ from above with $k^{2} (\log k + \log \log k)^{2}$.  This gives an upper bound of
\begin{equation}\label{mess}
  \frac{1 + 2 \cdot 10^{-25}}{\sqrt{2 \pi}} \frac{\sqrt{k}}{k^{2} (\log k + \log \log k)^{2}}
  \cdot \left(2 \pi e k (\log k + \log \log k)^{2}\right)^{k} \cdot \frac{e^{-4 \pi^{2} k (\log k + \log \log k)}}{1 - e^{(k-1)/s-2\pi/y}}.
\end{equation}
For $x < 0$, it is straightforward to see that
$\frac{1}{1 - e^{x}} < 1-\frac{1}{x}$, which gives
\[
\frac{1}{1 - e^{(k-1)/s - 2\pi/y}} \leq 1-\frac{1}{\frac{k-1}{s} - \frac{2 \pi}{y}} =
\frac{y}{2\pi - (k-1)\frac{y}{s}} + 1.
\]
We claim that this quantity is less than $\frac{y}{2\pi-1}$.  This can be seen by noting that $\frac{y}{2\pi - a} + 1 < \frac{y}{2\pi-1}$ exactly when $a < \frac{y+2\pi - 4\pi^2}{y+1-2\pi}$, and computing that $a = (k-1)\frac{y}{s} < 0.027$ and $\frac{y+2\pi - 4\pi^2}{y+1-2\pi} > 0.622$ when $k \geq 84$.

This gives an upper bound on~(\ref{mess}) of
\[
  \frac{1 + 2 \cdot 10^{-25}}{(2 \pi)^{3/2} (2 \pi - 1)}  \cdot \frac{1}{\sqrt{k}}
  \cdot \left(2 \pi e k\right)^{k} \left(\log k + \log \log k\right)^{2k-1} e^{-4 \pi^{2} k (\log k + \log \log k)^{2}}.
\]
Since $C > 1.806 (k/2 \pi)^{k} (\log k + \log \log k)^{k}$, this upper bound is less
than or equal to
\[
  \frac{1 + 2 \cdot 10^{-25}}{(2 \pi)^{3/2} (2 \pi - 1)} \cdot \frac{1}{\sqrt{k}}
  \cdot \frac{C}{1.806} (4 \pi^{2} e)^{k} (\log k + \log \log k)^{k-1} e^{-4 \pi^{2} k (\log k + \log \log k)},
\]
which can be rewritten as
\[
  \frac{C (1 + 2 \cdot 10^{-25})}{1.806 (2 \pi)^{3/2} (2 \pi - 1)}
  \cdot \frac{1}{\sqrt{k}} \cdot \frac{\left(4 \pi^{2} e\right)^{k} (\log k + \log \log k)^{k-1}}{(k \log k)^{4 \pi^{2} k}}.
\]
It is easy to see that the expression above is a decreasing function of $k$
and thus it is bounded by its value at $k = 84$, which is $\leq 2.534 \cdot 10^{-8294} C$.

We now compute an upper bound for the quantity
\[\sum_{n=s}^{\infty} 12 \sqrt{\log k} (C + 1.5^{k}) d(n) n^{(k-1)/2} e^{-2 \pi n /y}.\]
Using the facts that $1.5^{k} < 10^{-100} C$ for $k \geq 84$ and that $d(n) \leq 2 \sqrt{n}$, it suffices
to bound
\[
  24 (1 + 10^{-100}) \sqrt{\log k} \, C \sum_{n=s}^{\infty} n^{k/2} e^{-2 \pi n / y}.
\]
We apply Lemma~\ref{lemma3variation} with $\alpha = k/2$ and $\beta = -2\pi/y$, noting that the computation above showing that $s\beta < 1-k$ also shows that $s\beta < -k/2 = -\alpha$.  We obtain the upper bound
\begin{equation}
\label{cuspexp}
  24C \sqrt{\log k} (1 + 10^{-100}) s^{k/2} e^{-2 \pi s/y} \cdot \frac{1}{1 - e^{k/2s - 2 \pi / y}}.
\end{equation}
Performing a computation similar to that above, we find that $s^{k/2}e^{-2\pi s/y}$ is decreasing as a function of $s$, so we can bound it by replacing our choice of $s$ with $k^2(\log k+\log \log k)^2$.  Additionally, since $s > \frac{ky}{4 \pi}$ we may again use the fact that $\frac{1}{1 - e^{x}} < 1-\frac{1}{x}$ for $x < 0$ to obtain the bound
\[\frac{1}{1-e^{k/2s-2\pi/y}} \leq 1 - \frac{1}{\frac{k}{2s}-\frac{2\pi}{y}} =
y \cdot \left(\frac{2s + 4 \pi (s/y) - k}{4 \pi s - ky}\right) \leq y \left(\frac{2 + 4 \pi/y}{4 \pi - k(y/s)}\right).
\]
The quantity in parentheses above is a fraction whose numerator is a decreasing function of $k$, while the denominator is an increasing function of $k$. For $k = 84$,
\[
  \frac{2 + 4 \pi/y}{4 \pi - k(y/s)} < \frac{2}{4 \pi - 1}
\]
and thus for $k \geq 84$,
\[
\frac{1}{1-e^{k/2s-2\pi/y}} < \frac{2y}{4 \pi - 1}.
\]

Putting this all together, we obtain
an upper bound of
\begin{align*}
  &24.001 C \sqrt{\log k} (k (\log k + \log \log k))^{k}
  e^{-4 \pi^{2} k (\log k + \log \log k)} \cdot \frac{k (\log k + \log \log k)}{\pi(4 \pi-1)}\\
  &\leq 0.661C \sqrt{\log k} (k(\log k + \log \log k))^{k+1} e^{-4 \pi^{2} k (\log k + \log \log k)} \\
  &= 0.661C \sqrt{\log k} (k(\log k + \log \log k))^{k+1} \left(\frac{1}{k \log (k)}\right)^{4 \pi^{2} k}\\
  &= 0.661C k^{1 + k(1 - 4 \pi^{2})} (\log k)^{1/2 - 4 \pi^{2} k}
  (\log k)^{k+1} \left(1 + \frac{\log \log k}{\log k}\right)^{k+1}.
\end{align*}
Since $(1+x)^{k+1} \leq e^{(k+1)x}$ for $x > 0$, we set $x = (\log \log k)/(\log k)$ to obtain
\[
  \left(1 + \frac{\log \log k}{\log k}\right)^{k+1}
  \leq e^{\frac{(k+1) \log \log k}{\log k}} = (\log k)^{\frac{k+1}{\log k}}.
\]
This gives the bound
\begin{align*}
   & 0.661C k^{1 + k(1 - 4 \pi^{2})} (\log k)^{1/2 + \frac{k+1}{\log k} + k}\\
   &\leq 0.661C k^{1 + k(1 - 4 \pi^{2})} k^{3/2 + 2k} = 0.661C k^{5/2} k^{k(3 - 4 \pi^{2})}
\end{align*}
which is $\leq 1.91 \cdot 10^{-5892}C$ for $k \geq 84$.

Putting together these two bounds, we conclude that
\[
  \sum_{n=s}^{\infty} |a(n)| (e^{-2 \pi n/ y} - y^{k} e^{-2 \pi n y})
  \leq \sum_{n=s}^{\infty} |a(n)| e^{-2 \pi n/ y} < 2.534 \cdot 10^{-8294} C + 1.91 \cdot 10^{-5892} C < 0.00029C.
\]
This implies that
\[
  \sum_{n=\ell+1}^{s} a(n) (e^{-2 \pi n / y} - y^{k} e^{-2 \pi n y}) < 0
\]
and hence there is an integer $n \in [\ell+1,s]$ for which $a(n) < 0$.
\end{proof}

The next result shows that for weights $k \geq 92$, if enough of the initial coefficients of a form are nonnegative, then coefficients further out in the expansion are positive.

\begin{thm}
\label{92positivity}
If $f = 1 + \sum_{n=1}^{\infty} a(n) q^{n} \in M_{k}$ for $k \geq 92$ and $a(n) \geq 0$ for $n \leq k^{2} (\log k + \log \log k)^{2}$, then $a(n) > 0$ for $n \geq \frac{1}{7316} k^{4} (\log k + \log \log k)^{2}$.
\end{thm}
\begin{proof}
Since $\mathcal{N}(f) > k^{2} (\log k + \log \log k)^{2}$, the hypotheses of Lemma~\ref{smallC} are satisfied, and
Proposition~\ref{boundonC} implies that
\[
  C = \sum_{n=1}^{\ell} a(n) \leq 1.806 \left(\frac{k}{2 \pi}\right)^{k} (\log k + \log \log k)^{k}.
\]

Combining the upper bound on $C$ with Lemma~\ref{smallC} gives
\[|c(n)| \leq 12 \left(1.806 \left(\frac{k}{2\pi}\right)^{k} (\log k + \log \log k)^{k} + 1.5^{k}\right) \sqrt{\log k} \, d(n) n^{\frac{k-1}{2}}. \]
The contribution to $a(n)$ from the cusp form part will be dominated by the contribution from the Eisenstein series coefficient when $n$ is large enough, which will certainly happen if
\begin{equation}\label{ToProve}
\frac{-2k}{B_k} \sigma_{k-1}(n) > 12 \left(1.806 \left(\frac{k}{2\pi}\right)^{k} (\log k + \log \log k)^{k} + 1.5^{k}\right) \sqrt{\log k} \, d(n) n^{\frac{k-1}{2}}.
\end{equation}
We know that $\sigma_{k-1}(n) > n^{k-1}$ for $n > 1$ and that $d(n) \leq 2 \sqrt{n}$, which means that equation~(\ref{ToProve}) follows when
\[
\frac{-2k}{B_k}n^{k-1} > 24 \left(1.806 \left(\frac{k}{2\pi}\right)^{k} (\log k + \log \log k)^{k} + 1.5^{k}\right) \sqrt{\log k} \, n^{\frac{k}{2}},
\]
or
\[n^{\frac{k-2}{2}} > \frac{-24 B_k}{2k} (1.806 (k/2\pi)^{k} (\log k + \log \log k)^{k} + 1.5^{k}).
\]
Raising both sides to the $\frac{2}{k-2}$ power and simplifying, we obtain
\[
n > (24)^{\frac{2}{k-2}} \left(\frac{-B_k}{2k}\right)^{\frac{2}{k-2}}  \left(1.806 \left(\frac{k}{2\pi}\right)^{k} (\log k + \log \log k)^{k} + 1.5^{k}\right)^{\frac{2}{k-2}}.
\]
We replace terms on the right side of this inequality with larger terms to obtain a simpler lower bound on $n$ in terms of $k$ that will still imply~(\ref{ToProve}).

For $k = 92$, the quantity $1.806 (k/2\pi)^{k} (\log k + \log \log k)^{k}$ is approximately $1.934 \cdot 10^{179}$, while $1.5^k$ is approximately $1.5863 \cdot 10^{16}$.  It follows that
\[(1.806+\epsilon) \left(\frac{k}{2\pi}\right)^{k} (\log k + \log \log k)^{k} \geq 1.806 \left(\frac{k}{2\pi}\right)^{k} (\log k + \log \log k)^{k} + 1.5^k\]
for some small $\epsilon < 10^{-163}$ when $k \geq 92$.  Thus, we need an upper bound for
\[\left(1.806+\epsilon\right)^{\frac{2}{k-2}} \left(\frac{k}{2\pi}\right)^{2+\frac{4}{k-2}} (\log k + \log \log k)^{2+{\frac{4}{k-2}}}.\]

We compute that $(1.806+\epsilon)^{\frac{2}{k-2}} \leq {1.01323}$, that $k^{\frac{4}{k-2}} (2\pi)^{-2-\frac{4}{k-2}} \leq 0.02854$, and that $(\log k + \log \log k)^{\frac{4}{k-2}} \leq 1.08314$, since all of these are decreasing functions of $k$.  Thus, an appropriate upper bound is
\[(1.01323)(0.02854)(1.08314)k^2 (\log k + \log \log k)^2 \leq 0.031322 k^2 (\log k + \log \log k)^2.\]

For $k \geq 92$, we have $(24)^{\frac{2}{k-2}} \leq 24^{1/45}$.  Additionally, we have $(-B_k/2k) = (k-1)! \zeta(k)/(2\pi)^k$, so the factorial bound in~(\ref{facbound}) gives
\[\left(\frac{(k-1)!\zeta(k)}{(2\pi)^k}\right)^{\frac{2}{k-2}} \leq \left(2\pi(k-1)\right)^{\frac{1}{k-2}} \left(\frac{k-1}{e}\right)^{2+\frac{2}{k-2}} e^{\frac{1}{12(k-1)}} \zeta(k)^{\frac{2}{k-2}} (2\pi)^{-2-\frac{4}{k-2}}.\]
Bounding each piece on the right hand side for $k \geq 92$ shows that this quantity is bounded above by $0.004065946 k^2$.

Putting this all together, we find that $a(n) > 0$ if $n \geq 0.00013667 k^4 (\log k + \log \log k)^2$, or $n \geq \frac{1}{{7316}} k^4  (\log k + \log \log k)^2$.
\end{proof}

We now prove Theorem~\ref{Akbounds}.
\begin{proof}[Proof of Theorem~\ref{Akbounds}]
First, the lower bound on $A(k)$ was proven in Section~\ref{lowerbound}.

Suppose that \[f = 1 + \sum_{n=1}^{\infty} a(n) q^{n} = E_{k} + \sum_{n=1}^{\infty} c(n) q^{n}\] is a modular form in $M_{k}$ for which $\mathcal{N}(f)$, the smallest integer $r$ for which $a(r) < 0$ if such an integer exists, is as large as possible, so that $\mathcal{N}(f) = A(k)$.  We consider several cases.

$\bullet$ If $k \geq 92$ and $\mathcal{N}(f) \leq k^{2} (\log k + \log \log k)^{2}$, then the theorem holds because $k^{2} (\log k + \log \log k)^{2} < \frac{1}{7316} k^{4} (\log k + \log \log k)^{2}$ for $k \geq 92$.

$\bullet$ If $k \geq 92$ and $\mathcal{N}(f) > k^{2} (\log k + \log \log k)^{2}$, then Theorem~\ref{92positivity} implies that
$a(n) > 0$ for $n \geq \frac{1}{7316} k^{4} (\log k + \log \log k)^{2}$ and hence $A(k) = \mathcal{N}(f) \leq \frac{1}{7316} k^{4} (\log k + \log \log k)^{2}$.

$\bullet$ For $12 \leq k \leq 88$, the value of $A(k)$ given in Table~\ref{table2} is less than
$\frac{1}{7316}k^4 (\log{k} + \log \log{k})^2$.
\end{proof}

Next we prove Theorem~\ref{positivity}.

\begin{proof}[Proof of Theorem~\ref{positivity}]
For $k \geq 92$, Theorem~\ref{92positivity} gives
the desired result. For $k < 92$, we check that the value of $t$ in Table~\ref{table2} is less than $k^{2} (\log k + \log \log k)^{2}$.  Thus, if $f = 1 + \sum_{n=1}^{\infty} a(n) q^{n}$ satisfies the hypothesis of the theorem, then $a(n) \geq 0$ for $n \leq \ell$ and $a(t) \geq 0$.  Since the argument establishing inequality~(\ref{C2bound}) does not depend on any coefficients other than $a(1), \ldots, a(\ell)$ and $a(t)$ being nonnegative, we conclude that $a(n) > 0$ when $n \geq B(k)$, and it is straightforward to check that $B(k) < \frac{1}{7316} k^{4} (\log k + \log \log k)^{2}$ for $12 < k \leq 88$.
For $k = 12$, the calculations in Section~\ref{smallk} show
that if $a(1)$ and $a(2)$ are both nonnegative, then
$a(n) > 0$ for $n \geq 3$.
\end{proof}

Finally, we prove Theorem~\ref{upperbound}.

\begin{proof}[Proof of Theorem~\ref{upperbound}]
For $k \geq 84$, if the Fourier coefficients of a form $f(z)$ are nonnegative, then Proposition~\ref{boundonC} gives an upper bound on $C$.  Using this bound in Lemma~\ref{smallC} shows that if
$f = E_{k}(z) + \sum c(n) q^{n}$, then
\[
  |c(n)| \leq 12 \left(1.806 \left(\frac{k}{2\pi}\right)^{k} (\log k + \log \log k)^{k} + 1.5^{k}\right) \sqrt{\log k} \, d(n) n^{\frac{k-1}{2}}.
\]
It is straightforward to see that the function
\[
  \frac{\left(1.806 \left(\frac{k}{2\pi}\right)^{k} (\log k + \log \log k)^{k} + 1.5^{k}\right)\sqrt{\log k}}{\left(\frac{k}{2\pi}\right)^{k} (\log k)^{1.5k}}
\]
is decreasing and less than the constant stated in Theorem~\ref{upperbound} for $k \geq 84$.

For $20 \leq k \leq 84$, the bound $C_{2}$ on the cusp form part of $a(n)$ given in Table~\ref{table2} is less than
the bound stated in the theorem and so the result holds for those cases. For $k = 16$, the analysis in Section~\ref{smallk} shows that if $f = 1 + \sum_{n=1}^{\infty} a(n) q^{n}$ has all nonnegative coefficients, then $a(1) \leq 19338$, which implies that $C_{2} \leq 19333.5$, which is less
than the bound required by the theorem, namely $1.4245 \cdot 10^{7}$.
\end{proof}

\bibliographystyle{plain}
\bibliography{positive}

\begin{thebibliography}{10}

\bibitem{ChoJinLim}
Peter~J. Cho, Seokho Jin, and Subong Lim.
\newblock First sign changes of modular forms for general level {$N$}.
\newblock {\em J. Math. Anal. Appl.}, 523(2):Paper No. 127045, 13, 2023.

\bibitem{NIST:DLMF}
{\it NIST Digital Library of Mathematical Functions}.
\newblock http://dlmf.nist.gov/, Release 1.1.8 of 2022-12-15.
\newblock F.~W.~J. Olver, A.~B. {Olde Daalhuis}, D.~W. Lozier, B.~I. Schneider,
  R.~F. Boisvert, C.~W. Clark, B.~R. Miller, B.~V. Saunders, H.~S. Cohl, and
  M.~A. McClain, eds.

\bibitem{HeZhao}
Xiaoguang He and Lilu Zhao.
\newblock On the first sign change of {F}ourier coefficients of cusp forms.
\newblock {\em J. Number Theory}, 190:212--228, 2018.

\bibitem{Iwaniec}
Henryk Iwaniec.
\newblock {\em Topics in classical automorphic forms}, volume~17 of {\em
  Graduate Studies in Mathematics}.
\newblock American Mathematical Society, Providence, RI, 1997.

\bibitem{IK}
Henryk Iwaniec and Emmanuel Kowalski.
\newblock {\em Analytic number theory}, volume~53 of {\em American Mathematical
  Society Colloquium Publications}.
\newblock American Mathematical Society, Providence, RI, 2004.

\bibitem{JenkinsRouse}
Paul Jenkins and Jeremy Rouse.
\newblock Bounds for coefficients of cusp forms and extremal lattices.
\newblock {\em Bull. Lond. Math. Soc.}, 43(5):927--938, 2011.

\bibitem{KimmelPreprint}
Noam Kimmel.
\newblock Vanishing of {P}oincar{\'e} series for congruence subgroups.
\newblock {\em {\rm arXiv:2406.11302}}, 2024.

\bibitem{Lean}
The mathlib Community.
\newblock The lean mathematical library.
\newblock In {\em Proceedings of the 9th ACM SIGPLAN International Conference
  on Certified Programs and Proofs}, CPP 2020, page 367–381, New York, NY,
  USA, 2020. Association for Computing Machinery.

\bibitem{Robbins}
Herbert Robbins.
\newblock A remark on {S}tirling's formula.
\newblock {\em Amer. Math. Monthly}, 62:26--29, 1955.

\bibitem{Sturm}
Jacob Sturm.
\newblock On the congruence of modular forms.
\newblock In {\em Number theory ({N}ew {Y}ork, 1984--1985)}, volume 1240 of
  {\em Lecture Notes in Math.}, pages 275--280. Springer, Berlin, 1987.

\bibitem{Watson}
G.~N. Watson.
\newblock {\em A treatise on the theory of {B}essel functions}.
\newblock Cambridge Mathematical Library. Cambridge University Press,
  Cambridge, 1995.
\newblock Reprint of the second (1944) edition.

\end{thebibliography}

\end{document}